\documentclass[a4paper,11pt]{article}

\usepackage{amsmath}
\usepackage{amssymb}
\usepackage{amsthm}
\usepackage{graphicx}
\usepackage{amscd}
\usepackage{epic, eepic}
\usepackage{url}
\usepackage{color}
\usepackage[utf8]{inputenc} 
\usepackage{comment}
\usepackage{enumerate}   
\usepackage{graphicx}
\usepackage{epstopdf}
\usepackage{enumitem}
\usepackage{tikz}
\usepackage[top=24mm, bottom=25mm, left=25mm, right=25mm]{geometry}
\usepackage[bookmarks=false,hidelinks]{hyperref}

\newcommand{\floor}[1]{\left\lfloor{#1}\right\rfloor}

\newtheorem{theorem}{Theorem}

\newtheorem{lemma}{Lemma}
\usepackage{authblk}
\title{The maximum number of copies of an even cycle in a planar graph}
\author[1,3]{Zequn Lv}
\author[1]{Ervin Győri}
\author[1,3]{Zhen He}
\author[1]{Nika Salia}
\author[1]{Casey Tompkins}
\author[1,2]{Xiutao Zhu}
\date{}

\affil[1]{Alfr\'ed R\'enyi Institute of Mathematics.}
\affil[2]{Department of Mathematics, Nanjing University. }
\affil[3]{Department of Mathematical Sciences, Tsinghua University.}
\begin{document}

\maketitle

\begin{abstract}
We resolve a conjecture of Cox and Martin by determining asymptotically for every $k\ge 2$ the maximum number of copies of $C_{2k}$ in an $n$-vertex planar graph.
\end{abstract}

\section{Introduction}
A fundamental problem in extremal combinatorics is maximizing the number of occurrences of subgraphs of a certain type among all graphs from a given class.
In the case of $n$-vertex planar graphs, Hakimi and Schmeichel~\cite{hakimi} determined the maximum possible number of cycles length~$3$ and~$4$ exactly and showed that for any $k \ge 3$, the maximum number of $k$-cycles is $\Theta(n^{\floor{k/2}})$. 
Moreover, they proposed a conjecture for the maximum number of $5$-cycles in an $n$-vertex planar graph which was verified much later by Gy\H{o}ri \emph{et al.}~in~\cite{Gy}. 
The maximum number of $6$-cycles and $8$-cycles was settled asymptotically by Cox and Martin in~\cite{ccox}, and later the same authors~\cite{ccox2} also determined the maximum number of $10$-cycles and $12$-cycles asymptotically. 

Following the work of Hakimi and Schmeichel~\cite{hakimi}, Alon and Caro~\cite{alon} considered the general problem of maximizing copies of a given graph $H$ among $n$-vertex planar graphs. Wormald~\cite{Wormald} and later independently Eppstein~\cite{eppstein} showed that for $3$-connected $H$, the maximum number of copies of $H$ is $\Theta(n)$. The order of magnitude in the case when $H$ is a tree was determined in~\cite{generalized}, and the order of magnitude for an arbitrary graph was settled by Huynh, Joret and Wood~\cite{arb}. Note that by Kuratowski's theorem~\cite{Kuratowski} such problems can be thought of as generalized Tur\'an problems where we maximize the number of copies of the graph $H$ while forbidding all subdivisions of $K_5$ and $K_{3,3}$. 

Given that the order of magnitude of the maximum number of copies of any graph $H$ in an $n$-vertex planar graph is determined, it is natural to look for sharp asymptotic results.  While in recent times a number of results have been obtained about the asymptotic number of $H$-copies in several specific cases, less is known for general classes of graphs. Cox and Martin~\cite{ccox} introduced some general tools for studying such problems and conjectured that in the case of an even cycle $C_{2k}$ with $k\ge 3$, the maximum number of copies is asymptotically $n^k/k^k$.  We confirm their conjecture.

\begin{theorem}\label{mainconj}
For every $k\ge 3$, the maximum number of copies of $C_{2k}$ in an $n$-vertex planar graph is
\[\frac{n^k}{k^k} + o(n^k).
\]
\end{theorem}
A construction containing this number of copies of $C_{2k}$ is obtained by taking a $C_{2k}$ and replacing every second vertex by an independent set of approximately $n/k$ vertices, each with the same neighborhood as the original vertex. Cox and Martin~\cite{ccox} proved that an upper bound of $\frac{n^k}{k!}+o(n^k)$ holds and introduced a general method for maximizing the number of copies of a given graph in a planar graph. We will discuss this method in Section~\ref{reduction} and present another conjecture of Cox and Martin which implies Theorem~\ref{mainconj}. In Section~\ref{mainproof}, we prove this stronger conjecture (Theorem~\ref{thm1}). We have learned that Asaf Cohen Antonir and Asaf Shapira have independently obtained a bound within a factor of $e$ of the optimal bound attained in Theorem~\ref{thm1}.  

\section{Reduction lemma of Cox and Martin}
\label{reduction}
 
 For a positive integer $n$ we will consider functions $w:E(K_n)\to\mathbb{R}$ satisfying the conditions: 
 \begin{enumerate}
 \item For all $e\in E(K_n)$, $w(e)\ge 0$, \label{1}
 \item $\sum_{e\in E(G)} w(e) = 1$. \label{2}
 \end{enumerate} 
 For a subgraph $H'$ of $K_n$ and a function $w$ satisfying Conditions~\ref{1} and~\ref{2}, let  
 \[
 p_w(H') := \prod\limits_{e \in E(H')} w(e).
 \]
 Also for a fixed graph $H$ and $w$ satisfying Conditions~\ref{1} and~\ref{2} let
 \[
 \beta(w,H) := \sum_{H \cong H'\subseteq K_n} p_w(H').
 \]
For simplicity of notation, we will often omit statements about isomorphism in the sums.
Cox and Martin proved several reduction lemmas for pairs of graphs $H$ and $K$, in which an optimization problem involving $\beta(w,K)$ implies a corresponding upper bound on the maximum number of copies of the graph~$H$ among $n$-vertex planar graphs. 
We state the reduction lemma which Cox and Martin proved for cycles. 
For an integer $k\ge 3$, let 
\[\beta(k) = \sup_w \beta(w,C_k),\] 
where $w$ is allowed to vary across all~$n$ and all weight functions satisfying Conditions~\ref{1} and~\ref{2}. 

\begin{lemma}[Cox and Martin~\cite{ccox}]\label{redlemma}
For all $k\ge 3$, the number of $2k$-cycles in a planar graph is at most
\[
\beta(k)n^k+o(n^k).
\]
\end{lemma}

Cox and Martin conjectured that $\beta(k) \le \frac{1}{k^k}$.  By Lemma~\ref{redlemma} such a bound immediately implies Theorem~\ref{mainconj}.  In Section~\ref{mainproof}, we prove that this bound indeed holds.
\begin{theorem}\label{thm1}
For all $k\ge 3$,
\[
\beta(k) \leq \frac{1}{k^k}.
\]
Equality is attained only for weight functions satisfying  $w(e) = \frac{1}{k}$ for $e \in E(C)$ and $w(e) = 0 $ otherwise, where $C$ is a fixed cycle of length $k$ of $K_n$.
\end{theorem}

\section{Proof of Theorem~\ref{thm1}}\label{mainproof}

\newtheorem{definition}{Definition}[section]
\begin{proof}
Let us fix an integer $n$, a complete graph $K_n$ and a function $w$ satisfying Conditions~\ref{1} and~\ref{2}. Let us assume $w$ maximizes $\sum_{C_k \subseteq K_n} p_w(C_k)$.
Let~$P_j$ be a path with $j$ vertices. A $(j+2)$-vertex path with terminal vertices $u$ and $v$ is denoted by $vP_ju$. For vertices $u$ and $v$, a subgraph $H$ of $K_n$ and an integer $j$ such that  $2\leq j\leq n$, we define 
\[
f_H(j,u,v) = \sum_{u P_{j-2} v \subseteq H} p_w(u P_{j-2} v),
\]
and
\[
f_H(j,u) = \sum_{v \in V(H)\setminus\{u\}} f(j,u,v).
\]
In the case when $H$ is the complete graph $K_n$ we simply write $f(j,u,v)$ and $f(j,u)$.
The following lemma will be essential in the proof of Theorem~\ref{thm1}.
\begin{lemma}\label{lemma_main}
Let $k\ge 2$, and let $e_1=u_1v_1$ and $e_2=u_2v_2$ be distinct edges of $K_n$ such that $w(e_1)>0$ and $w(e_2)>0$. 
Then we have $f(k,u_1,v_1)= f(k,u_2,v_2).$
\end{lemma}
\begin{proof}[Proof of Lemma~\ref{lemma_main}]

We set $c:=w(e_1)+w(e_2)$ and define a function $g(x)$ in the following way:
\[
g(x):=\sum_{C_k \subseteq K_n} p_w(C_k) = A x(c-x) + B_1 x + B_2 (c-x) + C,
\]
where
\begin{align*}
    A &= \sum_{ \substack{C_k \subseteq K_n \\ e_1,e_2 \in C_k,}} \frac{p_w(C_k)}{w(e_1)w(e_2)},            \qquad C = \sum_{\substack{C_k \subseteq K_n \\e_1,e_2 \notin C_k}} {p_w(C_k)}, \\
  B_1 &= \sum_{\substack{C_k \subseteq K_n \\ e_1 \in C_k,e_2 \notin C_k}} \frac{p_w(C_k)}{w(e_1)},   \qquad B_2 = \sum_{\substack{C_k \subseteq K_n \\e_1 \notin C_k,e_2 \in C_k}} \frac{p_w(C_k)}{w(e_2)}.
\end{align*}

Note that  $\sum_{C_k \subseteq K_n} p_w(C_k)=g(w(e_1))$. Since $w$ maximizes the  function $\sum_{C_k \subseteq K_n} p_w(C_k)$, we have that the maximum of $g(x)$ is attained at $x=w(e_1)$  for $0\leq x\leq c$. Since  neither $w(e_1)\neq 0$ nor $w(e_1)\neq c$ we have $G_{t+1}(w(e_1)) = 0$. 
Hence we have $-2Ax + Ac + B_1 - B_2 = 0$ for $x = w(e_1)$. It follows that
\[
f(j,u_1,v_1) = B_1 + Aw(e_2) = B_2 + Aw(e_1) = f(j,u_2,v_2).\qedhere
\]
\end{proof}

From Lemma~\ref{lemma_main}, for an edge $uv$ with non-zero weight $w(uv) > 0$ we may assume  $f(j,u,v) = \mu$ for some fixed constant $\mu$. Hence we have 
\begin{equation}\label{Equation_miu}
  \sum_{C_k \subseteq K_n} p_w(C_k) = \frac{1}{k} \sum_{uv \in E(K_n)} w(uv)f(j,u,v) = \frac{\mu}{k} \sum_{uv \in E(K_n)} w(uv) = \frac{\mu}{k}.  
\end{equation}

Furthermore $w(e) \leq 1/k$ for every edge $e \in E(K_n)$. Indeed, 
\[
w(e)\mu = \sum_{e \in C_k} p_w(C_k) \leq \sum_{C_k \subseteq K_n} p_w(C_k) = \frac{\mu}{k}.
\]

For a vertex $v\in V(K_n)$ we denote $\sum_{u \in V(G)} w(uv)$ by $d_G(v)$. For a graph~$G$, a vertex set  $S\subseteq V(G)$ we denote the graph $G[V(G)\setminus S]$ by $G\setminus S$. Also  for an edge $e\in E(G)$, the graph with vertex set $V(G)$ and edge set $E(G)\setminus\{e\}$ is denoted by $G\setminus e$.

\begin{lemma}\label{thm2}
For a fixed integer $r$ such that $3\leq r\leq n$ and distinct vertices $v_1$ and $u$ 
there exists a sequence $v_2,v_3,\dots,v_{r-1}$ of distinct vertices such that 
\[
f_{G_1}(r,v_1,u) \leq d_{G_1}(v_1) d_{G_2}(v_2) \cdots d_{G_{t-1}}(v_{t-1}) f_{G_t}(r-t+1,v_t,u),
\]
for every integer $t$ satisfying $1\leq t \leq r-1$, where $G_1 = K_n \setminus v_1u$ and $G_i = K_n\setminus \{v_1,v_2,\dots,v_{i-1}\}$, for every $i = 2,3,\dots,r-1$.
\end{lemma} 

\begin{proof}
The proof proceeds by induction on $t$. The base case $t=1$ is trivial. We will prove the statement of the lemma for $t=j$ where $1<j\le r-1$ assuming that the statement holds for $t=j-1$.

We have
\[
f_{G_1}(r,v_1,u) \leq d_{G_1}(v_1) d_{G_2}(v_2) \cdots d_{G_{j-2}}(v_{j-2}) f_{G_{j-1}}(r-j+2,v_{j-1},u).
\]
Fix a vertex $v_j$ such that  $f_{G_j}(r-j+1,v_j,u) = \max_{x \in V(G_j)} f_{G_j}(r-j+1,x,u)$. 
Then,
\begin{align*}
&f_{G_{j-1}}(r-j+2,v_{j-1},u) = \sum_{x \in V(G_j)} w(v_{j-1}x) f_{G_j}(r-j+1,x,u) \\ 
&\leq \sum_{x \in V(G_j)} w(v_{j-1}x) f_{G_j}(r-j+1,v_j,u)= d_{G_{j-1}}(v_{j-1})f_{G_j}(r-j+1,v_j,u).
\end{align*}
Thus, we have
\[
f_{G_1}(r,v_1,u) \leq d_{G_1}(v_1) d_{G_2}(v_2) \cdots d_{G_{j-2}}(v_{j-2}) d_{G_{j-1}}(v_{j-1}) f_{G_{j}}(r-j+1,v_{j},u).\qedhere
\]
 \end{proof}

\begin{lemma}\label{thm3}
For every vertex $v$ and integer $r$ with  $2\leq r\leq n$, we have
\[
f(r,v) \leq \left( \frac{\sum_{e \in E(K_n)} w(e) }{r-1}\right)^{r-1}.
\]
\end{lemma}
\begin{proof}
We prove the lemma by induction on $r$. The base case $r=2$ is trivial since
$f(2,v) \le \sum_{e \in E(K_n)} w(e)$. 
We assume that the statement of the lemma holds for every $r$ satisfying  $2\leq r < j$ and prove it for $r=j$, where $2<j\le n$

We obtain
\begin{align*}
  f(j,v) &= \sum_{x \in V(K_n)} w(vx) f_{K_n \backslash\{v\} }({j-1},x)  
   \leq \sum_{x \in V(K_n)} w(vx) \left( \frac{\sum_{e \in E(K_n\backslash\{v\})} w(e) }{j-2}\right)^{j-2}\\
&\leq \left( \frac{\sum_{x \in V(K_n)} w(vx)  + (j-2)\left( \frac{\sum_{e \in E(K_n\backslash\{v\})} w(e) }{j-2}\right) }{j-1}\right)^{j-1}
= \left( \frac{\sum_{e \in E(K_n)} w(e) }{j-1}\right)^{j-1},
\end{align*}
where the first inequality comes from the induction hypothesis, and the second inequality follows from the inequality of the arithmetic and geometric means.  
\end{proof}

In order to finish the proof of Theorem~\ref{thm1} it is sufficient to show that  $\mu \leq \dfrac{1}{k^{k-1}}$ by~\eqref{Equation_miu}.

Choose an edge $v_0v_1$ with the maximum weight $w(v_0v_1)$. Let us denote the graph $K_n\setminus v_0v_1$ by $G_1$. 
By Lemma~\ref{thm2} we have a sequence of vertices $v_2,v_3,\dots,v_{k-1} \in V(K_n)$ satisfying the following inequality for every $t$:
\begin{equation}\label{Equation_f_g_1_k_}
f_{G_1}(k,v_1,v_0) \leq d_{G_1}(v_1) d_{G_2}(v_2) \cdots d_{G_{t-1}}(v_{t-1}) f_{G_t}({k-t+1},v_t,v_0),    
\end{equation}
where $1\leq t \leq k-1$,  $G_i = K_n\setminus\{v_1,v_2,\dots,v_{i-1}\}$, for all $i\in \{2,3,\dots,r-1\}$.
Here we distinguish the following two cases.

\noindent\textbf{Case 1:} Suppose that $d_{G_1}(v_1) + d_{G_2}(v_2) + \cdots + d_{G_{k-2}}(v_{k-2}) \leq \dfrac{k-2}{k}$.
Then by the inequality of the arithmetic and geometric means we have  
\[
\prod_{i=1}^{k-2} d_{G_i}(v_i) \leq 
\left( \frac{\sum_{i=1}^{k-2} d_{G_i}(v_i)}{k-2}\right)^{k-2} \leq \frac{1}{k^{k-2}}.
\]
From~\eqref{Equation_f_g_1_k_} we obtain the desired inequality
\[
\mu = f_{G_1}(k,v_1,v_0) \leq \left(\prod_{i=1}^{k-2} d_{G_i}(v_i)\right) \cdot f_{G_{k-1}}({2},v_{k-1},v_0)\leq \frac{1}{k^{k-2}} \frac{1}{k} \leq \frac{1}{k^{k-1}}.
\]
Even more the inequality holds with equality if and only if  $w(v_0v_1) = w(v_1v_2) = \cdots = w(v_{k-2}v_{k-1}) = w(v_{k-1}v_0) = 1/k$. Therefore equality in Theorem~\ref{thm1} is attained only for weight functions satisfying  $w(e) = \frac{1}{k}$ for $e \in E(C)$ and $w(e) = 0 $ otherwise, where $C$ is a fixed cycle of length $k$ of $K_n$.

\noindent\textbf{Case 2:} Suppose that $d_{G_1}(v_1) + d_{G_2}(v_2) + \cdots + d_{G_{k-2}}(v_{k-2}) > \dfrac{k-2}{k}$.
Let $t$ be the minimum integer in $\{1,2,\dots,k-2\}$ such that $d_{G_1}(v_1) + d_{G_2}(v_2) + \cdots + d_{G_{t}}(v_{t}) > t/k$. From minimality of~$t$ we have $d_{G_1}(v_1) + d_{G_2}(v_2) + \cdots + d_{G_{t-1}}(v_{t-1}) \leq (t-1)/k$. By the inequality of the arithmetic and geometric means we get
\begin{equation*}
\prod_{i=1}^{t-1} d_{G_i}(v_i) \leq \left( \frac{\sum_{i=1}^{t-1} d_{G_i}(v_i)}{t-1} \right)^{t-1} \leq \frac{1}{k^{t-1}}.
\end{equation*}
Observe that since the edge $v_0v_1$ has the maximum weight, by Lemma~\ref{thm3} we have

\begin{align*}
   &f_{G_t}({k-t+1},v_t,v_0) \leq \sum_{u \in V(G_{t+1})} w(v_tu)f_{G_{t+1}}({k-t},v_0,u)
   \leq  w(v_0v_1)f_{G_{t+1}}({k-t},v_0)\\&  \leq  w(v_0v_1) \left( \frac{\sum_{e \in E(G_{t+1})} w(e) }{k-t-1}\right)^{k-t-1}
    \leq
\left( \frac{w(v_0v_1) + \sum_{e \in E(G_{t+1})} w(e) }{k-t}\right)^{k-t},
\end{align*}
 where the last inequality follows from the inequality of the arithmetic and geometric means. By our choice of $t$, it follows that  
 \[w(v_0v_1) + \sum_{e \in E(G_{t+1})} w(e) \leq 1 - \sum_{i=1}^{t} d_{G_i}(v_i) < \dfrac{k-t}{k},\]
 and we obtain that
 \[
f_{G_t}({k-t+1},v_t,v_0) \leq
\left( \frac{w(v_0v_1) + \sum_{e \in E(G_{t+1})} w(e) }{k-t}\right)^{k-t} 
< \frac{1}{k^{k-t}}.
\]
Finally we have the desired bound on $\mu$:
\[
\mu = f(k,v_1,v_0) \leq \left(\prod_{i=1}^{t-1} d_{G_i}(v_i)\right) \cdot 
f_{G_t}({k-t+1},v_t,v_0) < \frac{1}{k^{t-1}}\frac{1}{k^{k-t}}= \frac{1}{k^{k-1}}. \qedhere
\]
\end{proof}

\section{Acknowledgements}
We would like to thank Ben Lund for some useful preliminary discussions on the topic. The research of Gy\H{o}ri and Salia was supported by the National Research, Development and Innovation Office NKFIH, grants  K132696 and SNN-135643. The research of Tompkins was supported by NKFIH grant K135800.

  \textit{E-mail addresses:} \\
  J.~Lv: \texttt{lvzq19@mails.tsinghua.edu.cn}\\
  E.~Gy\H{o}ri: \texttt{gyori.ervin@renyi.hu}\\
  Z.~He: \texttt{hz18@mails.tsinghua.edu.cn}\\
  N.~Salia: \texttt{nikasalia@yahoo.com}\\
  C.~Tompkins: \texttt{ctompkins496@gmail.com}\\
  X.~Zhu: \texttt{ zhuxt@smail.nju.edu.cn}.

\end{document}